\documentclass [12pt]{article}

\usepackage{a4wide}

\usepackage[utf8]{inputenc}  
\usepackage[T1]{fontenc}            
\usepackage[english, french, english]{babel} 
\usepackage{textcomp}
\usepackage{graphicx}
\usepackage{listing} 
\usepackage{amsthm}
\usepackage{amsmath}
\usepackage{amssymb}
\usepackage{amsfonts}
\usepackage{mathrsfs}
\usepackage{dsfont}
\usepackage{mathabx}
\usepackage{envmath}
\usepackage{cases}
\usepackage{comment}
\usepackage{epic, ecltree}
\usepackage{epsfig}
\usepackage{graphics}
\usepackage{color}
\usepackage{pgf,tikz,pgfplots} 
\usepackage{mathrsfs}
\usetikzlibrary{arrows}

\usepackage{fancyhdr}

\usepackage{lmodern}
\DeclareFontFamily{OMX}{lmex}{}
\DeclareFontShape{OMX}{lmex}{m}{n}{<->lmex10}{}

\usepackage[colorlinks=true, linkcolor=black, urlcolor=blue, filecolor=black, citecolor=black, menucolor=black]{hyperref}  
\usepackage{enumitem}

\newcommand{\ensemblenombre}[1]{\mathbb{#1}} 
\newcommand{\N}{\ensemblenombre{N}}

\newcommand{\C}{\ensemblenombre{C}}

\newcommand{\D}{\ensemblenombre{D}}

\newcommand{\Hol}{\operatorname{Hol}}

\newcommand{\dist}{\operatorname{dist}}
\newcommand{\eps}{\varepsilon}

\newcommand{\Ell}{\operatorname{Ell}}

\makeatletter
\newcommand{\bigint}{\@ifnextchar_\@bigintsub\@bigintnosub}
\def\@bigintsub_#1{\def\@int@subscript{#1}\@ifnextchar^\@bigintsubsup\@bigintsubnosup}
\def\@bigintsubsup^#1{\mathop{\text{\large$\int_{\text{\normalsize$\scriptstyle\@int@subscript$}}^{\text{\normalsize$\scriptstyle#1$}}$}}\nolimits}
\def\@bigintsubnosup{\mathop{\text{\large$\int_{\text{\normalsize$\scriptstyle\@int@subscript$}}$}}\nolimits}
\def\@bigintnosub{\@ifnextchar^\@bigintnosubsup\@bigintnosubnosup}
\def\@bigintnosubsup^#1{\mathop{\text{\large$\int^{\text{\normalsize$\scriptstyle#1$}}$}}\nolimits}
\def\@bigintnosubnosup{\mathop{\text{\large$\int$}}\nolimits}
\makeatother

\newtheorem{thm}{Theorem}[section]
\newtheorem*{thm*}{Theorem}
\newtheorem{prop}[thm]{Proposition}

\newtheorem{lem}[thm]{Lemma}

\theoremstyle{definition}

\newtheorem{ex}[thm]{Exemples}

\newcommand\blfootnote[1]{%
  \begingroup
  \renewcommand\thefootnote{}\footnote{#1}%
  \addtocounter{footnote}{-1}%
  \endgroup
}

\newcounter{qu}

\newcounter{qq}
\newcounter{ex}
\title{Identifying Bergman space functions from intervals}
\author{Andreas \bsc{Hartmann}\footnote{Université de Bordeaux, CNRS, Bordeaux INP, IMB, UMR 5251, 351 Cours de la Libération, F-33400, Talence, France, andreas.hartmann@math.u-bordeaux.fr}\, , Marcu-Antone \bsc{Orsoni}\footnote{Université de Laval, Pavillon Alexandre-Vachon, 1045 Avenue de la Médecine, Québec, QC, G1V 0A6, Canada, marcu-antone.orsoni@mat.ulaval.ca}}

\date{\today}

\usepackage{remreset}
\makeatletter \@removefromreset{figure}{subsection}\makeatother
\makeatletter \@removefromreset{figure}{section}\makeatother

\begin{document}

\maketitle

\begin{abstract}
We characterize functions of a Bergman space on a square by their values and derivatives on 
the diagonals. This problem is connected with the reachable space of the one-dimensional heat equation on 
a finite interval with boundary $L^2$-controls.
\end{abstract}



\blfootnote{This project is partially supported by ANR-24-CE40-5470, NSERC RGPIN-2026-06188 and NSERC DGECR-2026-00446.}

\section{Introduction}
A central question in sampling and interpolation theory is to reconstruct a function in a space of holomorphic functions from partial information. This information can be given by the values of the function on a sequence of points or on a more general set. In case the information is given on separated points, the problem gives rise to interpolating and sampling sequences. Such sequences have been extensively investigated in many classical spaces of holomorphic functions, such as Bergman, Fock, Dirichlet, Hardy spaces etc., see \cite{Seip} for a general reference on this. Dominating sets, for which the function is known on a measurable set, and more generally reverse Carleson measures, have also been studied in such spaces, see for instance the survey \cite{FHR} and references therein. In order to have stable reconstruction, like sampling sequences, dominating sets, or reverse Carleson measures, one in general needs some uniform distribution of the set (related to some densities) in the common domain of definition. In order to get rid of the requirement of uniform distribution one can increase the available information in a point by considering higher order derivatives. In the case of sampling and interpolation in the Fock space, this was considered for instance by Seip when the multiplicities at the sampling nodes are uniformly bounded  \cite{SeI,SW}, and later in \cite{BHKM} for unbounded multiplicities, see also \cite{ACHK} for related results in the Bergman space.

In general, when considering sampling, one has to assume {\it a priori} that the function is in the space under consideration. In this paper we follow a different path which can maybe be traced back to work by Luecking, and in particular to \cite{Lu} where the author considers integrals of higher order derivatives in connection with dominating sets. Yet, our main motivation is a paper by Aikawa-Hayashi-Saitoh  who considered 
the Bergman space on the sector $\Delta=\{z=x+iy:x>0,|y|<x\}$.  
In their work \cite{AHS}, the authors established the following representation formula relating the Bergman norm on $\Delta$ to an integral formula on $(0,\infty)$ involving all derivatives of a function: 
\begin{equation}\label{AHS}
\int_{\Delta}|f(z)|^2dA(z)=\sum_{n=0}^{\infty}\frac{2^n}{(2n+1)!}\int_0^{\infty}
 x^{2n+1}|f^{(n)}(x)|^2dx. \tag{AHS}
\end{equation}
See also \cite{AHSb} and \cite[Lemma 1]{HS} for more results in this vein.
As a result, it follows that for a $\mathcal{C}^\infty$-function 
on $(0,\infty)$ it can be decided whether it extends holomorphically to a Bergman space function in $\Delta$, and its norm is given by the formula (AHS). 
Remarkably, the usual {\it a priori} assumption in sampling that the function is in the space is not required. 

The result of \cite{AHS} can be reformulated in terms of a characterization of reachable states of the heat equation on the half-line $(0,\infty)$ controlled at 0 by weighted $L^2$-functions. As a consequence of \cite{AHS}, these reachable states are exactly functions in a suitably weighted Bergman space on $\Delta$. A natural question which arises now is whether an analogous result holds on a finite interval. This is related to a long-standing problem in control theory  about the characterization of the reachable states of the $1d$-heat equation on a finite interval $I$ with boundary $L^2$-control, a  question that goes back to the seminal work by Fattorini and Russell in the 70's \cite{FR} who established that functions holomorphic on a large enough horizontal strip centered on the interval are reachable. It was only in the 2010s that significant progress was made by Martin, Rosier and Rouchon \cite{MRR}, proving that reachable states extend holomorphically from the interval to a square $D$, one diagonal of which is $I$. They also showed that functions holomorphic in a suitable disk containing $D$ are reachable. Subsequently, Dard\'e and Ervedoza \cite{DE} proved that functions holomorphic in any neighborhood of $\overline{D}$ are reachable. It then became clear that spaces of holomorphic functions have to come into play. Some years later, in \cite{HKT}, it was shown that the reachable space is sandwiched between a Smirnov space and a Bergman space. Another important step was achieved in \cite{O}, where the second author identified the reachable space as a sum of unweighted Bergman spaces defined on sectors the intersection of which is exactly $D$. We refer also to \cite{KNT} for related work with weighted Bergman spaces. The final step was obtained in our work \cite{HO} where we solved a separation of singularities problem showing that the reachable space is exactly the Bergman space on $D$ as conjectured in \cite{HKT}.

Still, a central question remains open: as in the formula (AHS) on $(0,\infty)$, for a function defined on the interval only (with all its derivatives) can we decide whether it is in the Bergman space on $D$?
\\

In this paper, in order to discuss the above question, we will first give an alternative proof of the representation formula by Aikawa, Hayashi and Saitoh. While the proof in \cite{AHS} heavily builds on the heat equation and reproducing kernels, ours is essentially based on Parseval's formula. Indeed, taking a closer look at the integral formula they establish on $(0,\infty)$, one notices that the Parseval formula is lurking behind the scenes which allows one to relate the information on $|(x/\sqrt{2})^nf^{(n)}(x)|^2$ for a fixed $x>0$ to a weighted Bergman space norm on a disk $D_x=D(x,x/\sqrt{2})$ and obviously $\Delta=\bigcup_{x>0}D_x$.
\\

Our second aim is to adapt this observation to a finite interval, which we may assume to be $(0,1)$. We mention that the 
separation of singularities result established in \cite{HO}, which shows that $A^2(D)$ decomposes as the sum $A^2(\Delta)+A^2(1-\Delta)$, will not be of great help since we do not have explicit expressions for the decomposition $f=f_1+f_2$, where $f_1$ is in the Bergman space on $\Delta$ and $f_2$ on $1-\Delta$. One feature of our new proof of the (AHS)-formula is that it adapts to the new setting on the interval. However, a major obstacle occurs: the square $D$ cannot be covered by disks centered on one diagonal only implying that it is not possible to obtain an (AHS)-formula as such on a finite interval. Still, we will prove that such a formula holds when considering both diagonals of $D$.
\\

With these observations in mind, we now list the results we will establish concerning Bergman spaces on suitable domains symmetric with respect to $(0,1)$. We will need the following notation:
$$
\Omega=\bigcup_{0<x\le 1/2}D\left(x,x/\sqrt{2}\right)\cup \bigcup_{1/2<x<1}D\left(x,(1-x)/\sqrt{2}\right).
$$

We also need to replace the right hand side of \eqref{AHS} when considering the interval $(0,1)$ instead of $(0,+\infty)$. The following formula will be called (AHS)-type formula: 
\begin{eqnarray*}
\sum_{n=0}^{\infty}\frac{2^n}{(2n+1)!} \int_0^{1/2}
 x^{2n+1}\left(|f^{(n)}(x)|^2+|f^{(n)}(1-x)|^2 \right)dx. 
\end{eqnarray*}

\begin{enumerate}
\item We will first show that the (AHS)-type formula can be bounded by the (square of the) Bergman norm on $\Omega$.

\item We then show that there is a function $f$ for which this (AHS)-type formula is finite but $f$ is not in the Bergman space on $\Omega$, so that it can even not be expected that the (AHS)-type formula on $\Omega$ is comparable to the square of the Bergman norm on $\Omega$, which is a strictly smaller domain than $D$.

\item We estimate the (AHS)-type formula from {\it below} by the square of the norm in a Bergman space on a smaller set $\Omega_q$ depending on a certain parameter $q$ which we will introduce later.

\item Observing that the union of $\Omega_q$ and a suitably rotated version $\Omega'_q$ covers $D$, we finally show that functions in the Bergman space on $D$ can be recovered from the information on the two diagonals of $D$. 
We will actually give a proof of formula \eqref{MainEstim}  below independent of the previous result.
\end{enumerate}

The second item above shows that the situation on a finite interval is fundamentally different from the half-line case, and it shows that it is not possible to obtain the Bergman space on $D$ only from the information of $f$ on $(0,1)$ at least using an (AHS)-type formula. 

Letting $dA(z)=dxdy$ be the usual planar Lebesgue measure, our main result, which corresponds to item 4 above, reads as follows. 

\begin{thm*}
For every function $f\in A^2(D)$, we have
\begin{eqnarray}\label{MainEstim}
 \int_D |f(z)|^2dA(z)&=& 
\sum_{n=0}^{\infty}\frac{2^n}{(2n+1)!} \int_0^{1/2}
 x^{2n+1}\left(|f^{(n)}(x)|^2+|f^{(n)}(1-x)|^2+ \right. \nonumber\\
 & &\quad\left. |f^{(n)}(1/2+i(1/2-x))|^2+|f^{(n)}(1/2+i(x-1/2))|^2
\right)dx. 
\end{eqnarray}
Moreover, if for a function $f$ in $\mathcal{C}^{\infty}\big((0,1)\cup (1/2-i/2,1/2+i/2)\big)$ 
with $i^n \frac{\partial^n f}{\partial x^n}(1/2)= \frac{\partial^n f}{\partial y^n}(1/2)$ for all $n \in \N$,
the right hand side of \eqref{MainEstim} is finite, then $f$ extends to $A^2(D)$, and the identity \eqref{MainEstim} holds.
\end{thm*}

Note that we need to impose some local analyticity-type condition in $1/2$ relating the derivative in the horizontal direction with the derivative in the vertical direction of $f$ at the intersection point of the diagonals (for $n=1$, this condition corresponds to Cauchy-Riemann equations at the point  $1/2$, which is clearly not enough to conclude).\\ 

The theorem says in particular that a function can be recovered from the two diagonals. In other words, it can be seen from the behavior on the diagonals of a square whether a function in $\mathcal{C}^{\infty}$ on these diagonals and satisfying a certain analyticity condition at their intersection extends holomorphically to a square integrable function on the whole square, and the preceding remarks show that we cannot do better.
\\

The paper is organized as follows. In the next section we will present our proof of the (AHS)-formula. In Section 3 we will discuss the partial results 1.-3. presented above. The final Section 4 is devoted to the proof of the theorem stated above.
\\

A final word on notation. As usual, $A \lesssim B$ (respectively $A \gtrsim B$) means that there exists a constant $c>0 $ independent of the relevant variables such that $A \le c B$ (resp. $A \ge c B$), and $A \simeq B$ stands for $A \lesssim B$ and $A \gtrsim B$.

\subsection*{Acknowledgments}

We would like to thank the referees for their careful reading of the manuscript and for valuable suggestions which in particular helped improving the result of the main theorem. We also thank Karim Kellay who, after discovering our result in \cite{HO}, asked the question that led to our main result. 

\section{The Aikawa-Hayashi-Saitoh result}
For convenience, let us recall the Aikawa-Hayashi-Saitoh result.
\begin{thm*}\label{ThmAHS}
For every function $f\in A^2(\Delta)$, we have
\begin{equation}
\int_{\Delta}|f(z)|^2dA(z)=\sum_{n=0}^{\infty}\frac{2^n}{(2n+1)!}\int_0^{\infty}
 x^{2n+1}|f^{(n)}(x)|^2dx. \tag{AHS}
\end{equation}
Moreover, if the right hand side of (AHS) is bounded for a function $f\in \mathcal{C}^{\infty}(0,\infty)$, then $f$ extends analytically to a function in $A^2(\Delta)$, and the identity (AHS) holds.
\end{thm*}

We will now present the alternative proof of this result.

\begin{proof}
As already observed we can cover $\Delta$ by disks
$$
 \Delta=\bigcup_{x>0} D_x,\quad D_x=D\left(x,\frac{x}{\sqrt{2}}\right)=\{z\in\C:|z-x|<\frac{x}{\sqrt{2}}\}.
$$
Let $f\in \Hol(\Delta)$. In order to make the connection with the Parseval formula on $\partial D_x$ mentioned earlier, we write
\begin{eqnarray*}
 \varphi: \D&\longrightarrow &D_x\\
  u&\longmapsto& z=\frac{x}{\sqrt{2}}u+x,
\end{eqnarray*}
and
$$
 g(u)=f\circ\varphi(u)=f\left(\frac{x}{\sqrt{2}}u+x\right).
$$ 
The following formula is obvious:
\begin{eqnarray}\label{AHS0}
\sum_{n=0}^{\infty}\frac{2^n}{(2n+1)!}\int_0^{\infty}
 x^{2n+1}|f^{(n)}(x)|^2dx
&=&\int_0^{\infty}\sum_{n=0}^{\infty}\left| \frac{(x/\sqrt{2})^{n}f^{(n)}(x)}{n!} \right|^2
\frac{2^{2n}\,(n!)^{2}}{(2n+1)!}xdx
\end{eqnarray}
(which could be infinite {\textit{a priori}}).
Note that for all $n \in \N$,
\begin{equation}\label{Stirling}
\frac{2^{2n}\,(n!)^{2}}{(2n+1)!}= \frac{\Gamma(1+n)\,\Gamma\!\left(2-\frac{1}{2}\right)}{\Gamma\!\left(n+2-\frac{1}{2}\right)}.
\end{equation}
Indeed, using the fundamental property 
$$\Gamma(z+1)=z\Gamma(z), \quad \Re(z)>0, $$
and the known value $\Gamma\left(\frac{1}{2}\right)=\sqrt{\pi}$,  
we get by induction the value of $\Gamma$ on half-integers
$$\forall n \in \mathbb{N}, \ \Gamma\left(n+\frac{1}{2}\right)=
 \frac{(2n)!}{2^{2n}n!}\,\sqrt{\pi},
$$
from which we deduce \eqref{Stirling}. 
For later purposes we also recall the following asymptotic behavior using  Stirling's formula (see for instance \cite[p.8]{HKZ}):
\begin{equation}\label{Asymp}
\frac{2^{2n}\,(n!)^{2}}{(2n+1)!}= \frac{\Gamma(1+n)\,\Gamma\!\left(2-\frac{1}{2}\right)}{\Gamma\!\left(n+2-\frac{1}{2}\right)}
\simeq \frac{1}{\sqrt{n+1}}.
\end{equation}
Now, since $g(u)=f(\frac{x}{\sqrt{2}}u+x)$, we obtain for all $n \in \N$,
$$
 g^{(n)}(0)=\left(\frac{x}{\sqrt{2}}\right)^n f^{(n)}(x),
$$
and
$$
 \hat{g}(n):=\frac{g^{(n)}(0)}{n!}=\left(\frac{x}{\sqrt{2}}\right)^n \frac{f^{(n)}(x)}{n!}.
$$
Hence
\begin{eqnarray*}
\sum_{n=0}^{\infty}\frac{2^n}{(2n+1)!}\int_0^{\infty}
 x^{2n+1}|f^{(n)}(x)|^2dx
&=&\int_0^{\infty}\sum_{n=0}^{\infty}\left| \frac{(x/\sqrt{2})^{n}f^{(n)}(x)}{n!} \right|^2
\frac{2^{2n}(n!)^2}{(2n+1)!}\, xdx\\
&=&\int_0^{\infty}\sum_{n=0}^{\infty}|\hat{g}(n)|^2 \frac{\Gamma(1+n)\,\Gamma\!\left(2-\frac{1}{2}\right)}{\Gamma\!\left(n+2-\frac{1}{2}\right)} xdx.
\end{eqnarray*}
Here we recognize the norm of a weighted Bergman space in the unit disk (see for instance \cite[page 4]{HKZ}):

$$
\sum_{n=0}^{\infty}|\hat{g}(n)|^2 \frac{\Gamma(1+n)\,\Gamma\!\left(2-\frac{1}{2}\right)}{\Gamma\!\left(n+2-\frac{1}{2}\right)}= \int_{\D}|g(u)|^2(1-|u|^2)^{-1/2} \frac{dA(u)}{2\pi}
=:\|g\|_{A_{-\frac{1}{2}}^2}^2.
$$
Changing variables back, $u=\frac{\sqrt{2}}{x}(z-x)$,
\begin{eqnarray*}
 \int_{\D}|g(u)|^2(1-|u|^2)^{-1/2}\frac{dA(u)}{2\pi}&=&\int_{\D}|f\Big(\frac{x}{\sqrt{2}}u+x\Big)|^2
 (1-|u|^2)^{-1/2}\frac{dA(u)}{2\pi}\\
 &=&\int_{D_x}|f(z)|^2 \left(\frac{\sqrt{2}}{x}\right)^2
 (1-|\frac{\sqrt{2}}{x}(z-x)|^2)^{-1/2}\frac{dA(z)}{2\pi}\\
 &=&\int_{D_x}|f(z)|^2 \frac{\sqrt{2}}{x}
 \left(\frac{x^2}{{2}}-|z-x|^2\right)^{-1/2}\frac{dA(z)}{2\pi}.
\end{eqnarray*}

Hence, with $\chi_{D_x}$ being the characteristic function of $D_x$,
\begin{eqnarray}
\lefteqn{\sum_{n=0}^{\infty}\frac{2^n}{(2n+1)!}\int_0^{\infty}
 x^{2n+1}|f^{(n)}(x)|^2dx \notag}\\
 &=& \int_0^{\infty} \int_{D_x}|f(z)|^2 \frac{\sqrt{2}}{x}
 \left(\frac{x^2}{{2}}-|z-x|^2\right)^{-1/2}\frac{dA(z)}{2\pi}xdx\nonumber\\
&=& \sqrt{2} \int_0^{\infty} \int_{D_x}|f(z)|^2  \left(\frac{x^2}{{2}}-|z-x|^2\right)^{-1/2}\frac{dA(z)}{2\pi}{dx}\nonumber\\
&=& \sqrt{2} \int_0^{\infty}\int_{\Delta}\chi_{D_x}(z)|f(z)|^2 \left(\frac{x^2}{{2}}-|z-x|^2\right)^{-1/2}\frac{dA(z)}{2\pi}{dx}\nonumber\\
&=& \sqrt{2} \int_{\Delta}|f(z)|^2\int_{0}^{\infty}\chi_{D_x}(z) \left(\frac{x^2}{{2}}-|z-x|^2\right)^{-1/2}{dx} \frac{dA(z)}{2\pi}.\label{MainFormula}
\end{eqnarray}

Let us introduce a notation:
$$
K_{\infty}(z)=\frac{\sqrt{2}}{2\pi}\int_{0}^{\infty}\chi_{D_x}(z) \left(\frac{x^2}{{2}}-|z-x|^2\right)^{-1/2}{dx}.
$$
Our aim will be to show that $K_{\infty}(z)=1$ for every $z\in\Delta$.
Observe first that $\chi_{D_x}(z)=1$ if and only if $|z-x|^2<x^2/2$, which is equivalent to
$$
 |z|^2-2x\Re z+x^2<x^2/2.
$$ 
This can be rewritten as a second degree inequality in $x$: 
\begin{equation}\label{2ndDeg}
x^2-4x\Re z+2|z|^2<0.
\end{equation}
Setting $z=a+ib$, $a>0$ and $|b|<a$, we get \eqref{2ndDeg} exactly when $x\in (x_1,x_2)$ where
$$
x_{1/2}=\frac{4\Re z\pm 2\sqrt{2(a^2-b^2)}}{2}=2a\pm \sqrt{2(a^2-b^2)}.
$$
We thus have to estimate
$$
K_{\infty}(z)=\frac{\sqrt{2}}{2\pi}\int_{0}^{\infty}\chi_{D_x}(z) \left(\frac{x^2}{{2}}-|z-x|^2\right)^{-1/2}{dx}
=\frac{\sqrt{2}}{2\pi}\int_{x_1}^{x_2}\left(\frac{x^2}{{2}}-|z-x|^2\right)^{-1/2}{dx}.
$$
Note that
\begin{eqnarray}\label{fct.h}
 h(x)&:=&\frac{x^2}{{2}}-|z-x|^2=\frac{x^2}{2}-|z|^2+2x\Re z-x^2
 =-\frac{1}{2}(x^2-4x\Re z+2|z|^2)\nonumber\\
&=&\frac{1}{2}(x-x_1)(x_2-x),
\end{eqnarray}
which is exactly the polynomial discussed previously, and one over the square-root of
which we integrate now between the roots ensuring the positivity of the polynomial.
Hence
$$
\int_{x_1}^{x_2}\left(\frac{x^2}{{2}}-|z-x|^2\right)^{-1/2}{dx}
=\sqrt{2}\int_{x_1}^{x_2}\frac{1}{\sqrt{x-x_1}}\frac{1}{\sqrt{x_2-x}}dx.
$$
Let us make the change of variables transforming $[x_1,x_2]$ to $[-1,1]$,
$$
 y=\frac{2}{x_2-x_1}(x-x_1)-1
$$
so that
$$
 x-x_1=(y+1)\frac{x_2-x_1}{2}, \quad   \quad x-x_2=(y-1)\frac{x_2-x_1}{2}.
$$
Then, noting that $dx=\frac{x_2-x_1}{2}dy$,
\begin{eqnarray*}
\int_{x_1}^{x_2}\frac{1}{\sqrt{x-x_1}}\frac{1}{\sqrt{x_2-x}}dx.
&=&\int_{-1}^1\frac{1}{\sqrt{y+1}} \frac{1}{\sqrt{1-y}}\frac{2}{x_2-x_1}\times  \frac{x_2-x_1}{2}dy\\
&=&\int_{-1}^1\frac{1}{\sqrt{1-y^2}}dy\\
 &=&\pi.
\end{eqnarray*}
Hence
\begin{eqnarray}\label{IntChi}
K_{\infty}(z)
&=&\frac{\sqrt{2}}{2\pi}\int_{x_1}^{x_2}\left(\frac{x^2}{{2}}-|z-x|^2\right)^{-1/2}{dx}\nonumber\\
&=& \frac{1}{\pi}\int_{x_1}^{x_2}\frac{1}{\sqrt{x-x_1}}\frac{1}{\sqrt{x_2-x}}dx\nonumber\\
&=& 1.
\end{eqnarray}
We conclude the (AHS) formula:
\begin{eqnarray*}
{\sum_{n=0}^{\infty}\frac{2^n}{(2n+1)!}\int_0^{\infty}
 x^{2n+1}|f^{(n)}(x)|^2dx}
&=&\sqrt{2}\int_{\Delta}|f(z)|^2\int_{0}^{\infty}\chi_{D_x}(z) \left(\frac{x^2}{{2}}-|z-x|^2\right)^{-1/2}{dx} \frac{dA(z)}{2\pi}\\
&=&\int_{\Delta}|f(z)|^2dA(z).
\end{eqnarray*}
It follows from the above computations that when $f\in A^2(\Delta)$, the right hand side of (AHS) is finite and indeed equal to the square of the $A^2(\Delta)$-norm of $f$.
\\

We now  prove  that if 
\begin{equation}\label{CvInt}
\sum_{n=0}^{\infty}\frac{2^n}{(2n+1)!}\int_0^{\infty}
 x^{2n+1}|f^{(n)}(x)|^2dx<\infty,
\end{equation}
for some function $f\in\mathcal{C}^{\infty}(0,\infty)$, then $f$ extends to a holomorphic function in $\Delta$ and belongs to $A^2(\Delta)$. 
We will need the following auxiliary result.
\begin{lem}\label{RealAnal}
Let $f\in\mathcal{C}^{\infty}(0,\infty)$ such that 
$$
\sum_{n=0}^{\infty}\frac{2^n}{(2n+1)!}\int_0^{\infty}
 x^{2n+1}|f^{(n)}(x)|^2dx<\infty.
$$
Then $f$ is real analytic on $(0,\infty)$.
\end{lem}

\begin{proof}
Pick $f\in \mathcal{C}^{\infty}(0,\infty)$. Let $x_0>0$ be arbitrary, and let $\eps>0$ be fixed later. Since $f$ is $\mathcal{C}^{\infty}$ on $(0,\infty)$, for every $n\in\N$ we can expand $f$ as a Taylor series with integral remainder
$$
f(x)=\sum_{k=0}^n\frac{f^{(k)}(x_0)}{k!}
 (x-x_0)^k+\underbrace{\int_{x_0}^x\frac{(x-t)^n}{n!} f^{(n+1)}(t)dt.}_{R_n(x)}
$$
We have to show that, when $n$ tends to infinity, $R_n$ tends uniformly to zero on a sufficiently small interval $(x_0-\eps,x_0+\eps)$, $0<\eps<x_0$. Assuming first that $x_0<x<x_0+\eps$, we have
\begin{eqnarray*}
\left|\int_{x_0}^x\frac{(x-t)^n}{n!} f^{(n+1)}(t)dt\right|^2
 &\le& \int_{x_0}^x\frac{|x-t|^{2n}}{(n!)^2}dt \times \int_{x_0}^x | f^{(n+1)}(t)|^2dt\\
 &\le&\frac{\eps^{2n+1}}{(n!)^2(x_0-\eps)^{2n+3}}
  \times \int_{x_0}^xt^{2n+3}|f^{(n+1)}(t)|^2dt\\
&\le&\frac{\eps^{2n+1}(2n+3)!}{(n!)^2(x_0-\eps)^{2n+3}2^{n+1}}
  \times \underbrace{\frac{2^{n+1}}{(2n+3)!} \int_{0}^\infty t^{2n+3}|f^{(n+1)}(t)|^2dt}_{\gamma_n}.
\end{eqnarray*}
(The estimates work similarly for $x_0-\eps<x<x_0$: indeed, the factor $(x_0-\eps)^{2n+3}$ in the denominator allows to take care also of this second case, since then $t\ge x_0-\eps$.)
By assumption $(\gamma_n)$ is summable, so that it tends to zero. It remains to examine the first factor, which can be written as
\begin{eqnarray*}
\frac{\eps^{2n+1}(2n+3)!}{(n!)^2(x_0-\eps)^{2n+3}2^{n+1}}
 &=&\left|\frac{\eps\sqrt{2}}{x_0-\eps}  \right|^{2n}
  \times  \frac{\eps (2n+3)(2n+2)}{2|x_0-\eps|^3}
   \times \frac{(2n+1)!}{(n!)^22^{2n}}
\end{eqnarray*}
By \eqref{Asymp}, the last factor behaves like $\sqrt{n+1}$. 
In order to have uniform convergence of the remainder term, it is thus sufficient to have
$$
 \left|\frac{\eps}{x_0-\eps}\right|<\sqrt{\frac{1}{2}},
$$
which holds for instance when 
$$
 0<\eps<\frac{x_0}{\sqrt{2}+1}
$$
\end{proof}

Now, the convergence of the integral in \eqref{CvInt} implies in particular that there exists a measurable set $E\subset (0,\infty)$ of full measure such that for every $x\in E$, we have
$$
\sum_{n=0}^{\infty}\frac{2^n}{(2n+1)!}
 x^{2n+1}|f^{(n)}(x)|^2<\infty.
$$
From this we deduce that
$$
\frac{|f^{(n)}(x)|}{n!}\lesssim\sqrt{\frac{(2n+1)!}{2^{2n}(n!)^2}}\times \left(\frac{\sqrt{2}}{x}\right)^n\times\frac{1}{\sqrt{x}}.
$$
Therefore, using \eqref{Asymp}, we obtain that for all $x\in E$, the function defined by 
$$
 f_{x}(z):=\sum_{n=0}^{\infty}\frac{f^{(n)}(x)}
  {n!}(z-x)^n
$$
is holomorphic in $D_x$ and, by Lemma \ref{RealAnal}, agrees with $f$ on a neighborhood of $x$ in $(0,+\infty)$. Since $f$ is real analytic on $(0, \infty)$, the two functions $f$ and $f_x$ agree on $D_x\cap (0,+\infty)$.
Observe now that if $x_1,x_2\in E$ are sufficiently close to each other so that $D_{x_1}$ and $D_{x_2}$ intersect, then $f_{x_1}$ and $f_{x_2}$ coincide with $f$ on a real interval with positive length. 
Hence, they represent the same function.
Finally, as $E$ is dense and $x \mapsto \frac{\sqrt{x}}{2}$ is continuous, we obtain that $
\Delta=\bigcup_{x>0}D_x=\bigcup_{x\in E}D_x$, so $f$ extends holomorphically to $\Delta$. The (AHS)-formula then shows that $f$ is square-integrable on $\Delta$ and hence $f\in A^2(\Delta)$.
\end{proof}

We mention that $f$ could {\it a priori} be the restriction to $(0,\infty)$ of a non analytic function, like for instance an anti-analytic or a sum of an analytic and an anti-analytic function. The theorem states that \eqref{AHS} holds for the analytic extension of $f$.

\section{The case of one diagonal}
In view of the discussions in the previous section, let us recall the following domain
$$
\Omega=\bigcup_{0<x\le1/2}D_x\cup \bigcup_{0<x < 1/2} \left(1-D_x \right)
$$

\begin{figure}[h!]
\begin{center}
\includegraphics[scale=1]{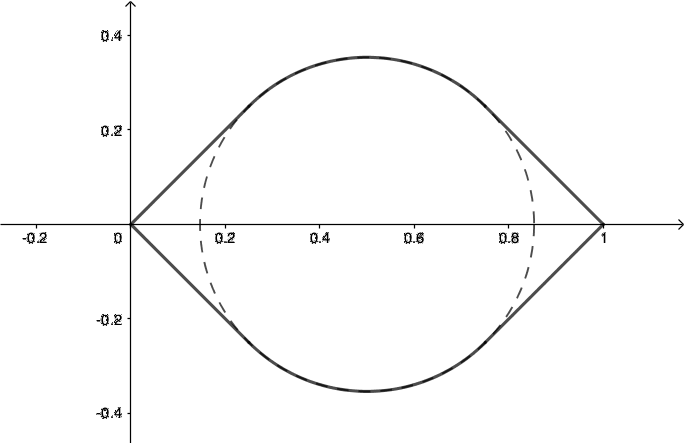}
\end{center}
\caption{The domain $\Omega$.\label{Fig1}}
\end{figure}

In this section we will see that even on the domain $\Omega$ which is strictly smaller than $D$ we cannot decide whether a function $f\in\mathcal{C}^{\infty}(0,1)$ extends to a function in $A^2(\Omega)$ solely based on its behavior on $(0,1)$ as reflected by an (AHS)-type formula (see the left hand side of \eqref{AHStypeSquare} below). Still, we have an upper bound:
\begin{prop}\label{Prop1}
For every function $f\in A^2(\D)$, we have
\begin{align}\label{AHStypeSquare}
\sum_{n=0}^{\infty}\frac{2^n}{(2n+1)!} &\int_0^{1/2}
 x^{2n+1}\left(|f^{(n)}(x)|^2+|f^{(n)}(1-x)|^2 \right)dx \nonumber\\
&\le 2 \int_\Omega |f(z)|^2dA(z) \le 2 \int_{D} |f(z)|^2dA(z).
\end{align}
\end{prop}

It follows from the proof of the main theorem presented in Section \ref{S4} that the constant 2 can actually be replaced by 1. At this step, we prefer keeping the simpler argument below.

\begin{proof}
Exactly as in \eqref{MainFormula}, in which we  replace integration on $(0,\infty)$ by integration on $(0,1/2)$, 
and observing that for every $x\in (0,1/2)$, $D_x\subset \Omega$, we get
\begin{eqnarray}
\lefteqn{\sum_{n=0}^{\infty}\frac{2^n}{(2n+1)!}\int_0^{1/2}
 x^{2n+1}|f^{(n)}(x)|^2dx}\nonumber\\
&&=\frac{\sqrt{2}}{2\pi}\int_0^{1/2}\int_{\Omega}\chi_{D_x}(z)|f(z)|^2 \left(\frac{x^2}{{2}}-|z-x|^2\right)^{-1/2}dA(z){dx}\nonumber\\
&&=\frac{\sqrt{2}}{2\pi}\int_{\Omega}|f(z)|^2\int_{0}^{1/2}\chi_{D_x}(z) \left(\frac{x^2}{{2}}-|z-x|^2\right)^{-1/2}{dx} dA(z)
\label{IdOmega}.
\end{eqnarray}

Note that we cannot integrate beyond $1/2$ since this would imply that the function is 
holomorphic on a domain containing points outside the square $D$.
Now, by \eqref{IntChi}, we have
\begin{align*}
K_{1/2}(z)&:=\frac{\sqrt{2}}{2\pi}\int_{0}^{1/2}\chi_{D_x}(z) \left(\frac{x^2}{{2}}-|z-x|^2\right)^{-1/2}{dx}\\
&\le\frac{\sqrt{2}}{2\pi}\int_{0}^{\infty}\chi_{D_x}(z) \left(\frac{x^2}{{2}}-|z-x|^2\right)^{-1/2}{dx}
=1.
\end{align*}

Using a symmetry argument we similarly obtain the estimate of the part $|f^{(n)}(1-x)|^2$ and thus the desired result.
\end{proof}

We will now show that the above proposition is optimal in the sense that the inequality cannot be reversed in general.
\begin{prop}
\label{Prop-counterexample}
There exists a function $f\in \Hol(\Omega)$, such that $f\not\in A^2(\Omega)$ and
\begin{eqnarray*}
\sum_{n=0}^{\infty}\frac{2^n}{(2n+1)!} \int_0^{1/2}
 x^{2n+1}\left(|f^{(n)}(x)|^2+|f^{(n)}(1-x)|^2 \right)dx <\infty.
\end{eqnarray*}
\end{prop}

\begin{proof}    
Let 
$$
 z_0=\frac{1}{2}+i\frac{1}{2\sqrt{2}}
$$
which is one of the two boundary points of $\Omega$ with real part $1/2$. 
Then clearly the function $f$ defined by
$$
 f(z)=\frac{1}{z_0-z}
$$
is not in $A^2(\Omega)$. This is a quite well-known fact (for instance, functions in $A^2(\Omega)$ grow like $o(1/\dist(z,\partial\Omega))$, which is not the case of $f$). Still, let us include a simple argument. Since $\Omega\cap D(z_0,1/4)$ contains a Stolz angle 
\begin{align*}\Gamma_{z_0}&=\{z: \ |\arg(z_0-z)-\pi/2|<\pi/4, \, |z-z_0|<1/4\}\\
&=\left\lbrace z_0 - re^{it}: \ \frac{\pi}{4} < t < \frac{3\pi}{4},\  0<r<\frac{1}{4}\right\rbrace,
\end{align*}
we get
$$
 \int_{\Omega}|f(z)|^2dA(z)\ge \int_{\pi/4}^{3\pi/4}\int_0^{1/4}|f(z_0-re^{it})|^2 rdrdt
 =\frac{\pi}{2}\int_0^{1/4}\frac{1}{r}dr=\infty.
$$

Now
$$
 f^{(n)}(z)=\frac{n!}{(z_0-z)^{n+1}},\quad z\in\Omega.
$$
Recall from \eqref{Asymp}:
$$
\frac{2^{2n}\,(n!)^{2}}{(2n+1)!}= \frac{\Gamma(1+n)\,\Gamma\!\left(2-\frac{1}{2}\right)}{\Gamma\!\left(n+2-\frac{1}{2}\right)}
\simeq \frac{1}{\sqrt{n+1}}.
$$
Then
\begin{eqnarray*}
\sum_{n=0}^{\infty} \frac{2^n}{(2n+1)!}\int_0^{1/2}x^{2n+1}|f^{(n)}(x)|^2dx
 &=&\sum_{n=0}^{\infty} \frac{2^n}{(2n+1)!}
 \int_0^{1/2}x^{2n+1}\frac{(n!)^2}{|x-z_0|^{2(n+1)}}dx\\
&=&\sum_{n=0}^{\infty} \frac{2^n (n!)^2}{(2n+1)!}
  \int_0^{1/2}\frac{x^{2n+1}}{|x-z_0|^{2(n+1)}}dx\\
&\simeq&\sum_{n=0}^{\infty} \frac{1}{2^n\sqrt{n+1}}
  \int_0^{1/2}\frac{x^{2n+1}}{|x-z_0|^{2(n+1)}}dx\\
&\simeq&\sum_{n=0}^{\infty} \frac{1}{\sqrt{n+1}}
  \int_0^{1/2}\frac{(x/\sqrt{2})^{2n}}{|x-z_0|^{2n}}xdx\\
\end{eqnarray*}

It is easy to see that for all $x \in [0, \, 1/2]$, 
$$
 \frac{(x/\sqrt{2})^2}{|x-z_0|^2}
 =\frac{x^2}{2\Big((x-1/2)^2+1/8\Big)}
 =\frac{4x^2}{8x^2-8x+3}=2x\times \frac{2x}{8x^2-8x+3}
 \le 2x.
$$

Hence
$$
\int_0^{1/2}\frac{(x/\sqrt{2})^{2n}}{|x-z_0|^{2n}}xdx\le \int_0^{1/2}(2x)^ndx
=\frac{1}{2(n+1)},
$$
so that the above series converges. By symmetry, the same holds true for the series corres\-ponding to $|f^{(n)}(1-x)|^2$. This ends the proof.
\end{proof}

We next show that if we seek for identification from only one diagonal then we can get a lower estimate on any smaller domain $\Omega_q$ of the following form.
$$
 \Omega_q=\Bigl(\Omega \cap \left(\{ 0<\Re z<1/4\} \cup \{3/4<\Re z<1\}\right)\Bigl) \cup \bigl(\Omega \cap \operatorname{Ell}_{q,0}\bigr),
$$
where $\Ell_{q,0}$ is an ellipse of the form (see Figure \ref{Fig2}) 
$$
 \Ell_{q,0}=\left\lbrace z=x+iy: q\left(x-\frac{1}{2}\right)^2+(16-q)y^2<1\right\rbrace,\quad q\in (0,8).
$$
Note that for $q=8$ this is exactly the disk $D(1/2,\frac{1}{2\sqrt{2}})$.  
Observe also that the ellipses
$\Ell_{q,0}$ fix the points $(1/4,\pm 1/4)$, $(3/4,\pm 1/4)$, 
and satisfy that $\Ell_{q,0}\cap\{z:1/4<\Re z<3/4\}$ is contained in $\Omega$. Moreover, $\bigcup_{0<q<8}\Omega_q=\Omega$.

\begin{figure}[h!]
\begin{center}
\includegraphics[scale=1]{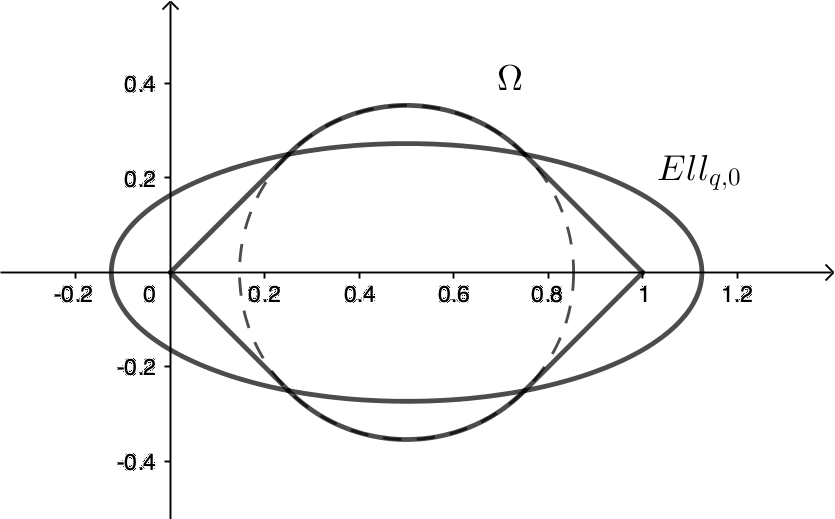}
\end{center}
\caption{The domain $\Omega$ and the ellipse $\Ell_{q,0}$.\label{Fig2}}
\end{figure}

\begin{prop}\label{Prop2}
For every $q\in (0,8)$ there is a constant $C_q$ such that each function $f\in\mathcal{C}^{\infty}(0,1)$ for which
\begin{eqnarray}\label{AHSintOmega}
\sum_{n=0}^{\infty}\frac{2^n}{(2n+1)!} \int_0^{1/2}
 x^{2n+1}\left(|f^{(n)}(x)|^2+|f^{(n)}(1-x)|^2 \right)dx <\infty
\end{eqnarray}
extends to a holomorphic function on $\Omega$, and we have
\begin{eqnarray}\label{AHSint}
\sum_{n=0}^{\infty}\frac{2^n}{(2n+1)!} \int_0^{1/2}
 x^{2n+1}\left(|f^{(n)}(x)|^2+|f^{(n)}(1-x)|^2 \right)dx 
\ge C_q \int_{\Omega_q} |f(z)|^2dA(z)
\end{eqnarray}
\end{prop}

\begin{proof}
As in Lemma \ref{RealAnal} we deduce from \eqref{AHSintOmega} that $f$ is real analytic (there is no problem at $x=1/2$ since we actually integrate on $(0,1)$ against the weight $\dist(x,\{0,1\})^{2n+1}$), and the arguments after that Lemma yield that $f$ extends holomorphically to $\Omega$.
\\

Recall now the identity \eqref{IdOmega} 
\begin{eqnarray*}
\sum_{n=0}^{\infty}\frac{2^n}{(2n+1)!}\int_0^{1/2}
 x^{2n+1}|f^{(n)}(x)|^2dx
=\frac{\sqrt{2}}{2\pi}\int_{\Omega}|f(z)|^2\int_{0}^{1/2}\chi_{D_x}(z) \left(\frac{x^2}{{2}}-|z-x|^2\right)^{-1/2}{dx} dA(z).
\end{eqnarray*}
It is thus enough to estimate from below on $\Omega_q$ the expression
$$
K_{1/2}(z)=\frac{\sqrt{2}}{2\pi}\int_{0}^{1/2}\chi_{D_x}(z) \left(\frac{x^2}{{2}}-|z-x|^2\right)^{-1/2}{dx}.
$$
Again $\chi_{D_x}(z)=1$ if and only if $x\in (x_1,x_2)$ and
$$
x_{1/2}=\frac{4\Re z\pm 2\sqrt{2(a^2-b^2)}}{2}=2a\pm \sqrt{2(a^2-b^2)}.
$$

For the lower bound we will first decompose $\Omega_q$ into two pieces $\Omega_1\cup \Omega_2$ where
$\Omega_1=\Omega_q\cap \{0<\Re z<1/2\}$ and $\Omega_2=\Omega
_q\setminus\Omega_1$.
By symmetry it is enough to discuss the estimate on $\Omega_1$.
We will decompose again $\Omega_1$ into two pieces 
$$
 \Omega_1=\Omega_{11}\cup\Omega_{12}
$$
where $\Omega_{11}=\Omega\cap\{0<\Re z <1/4\}$, $\Omega_{12}=\Ell_{q,0}\cap
\{1/4\le\Re z<1/2\}$.

\begin{figure}[h!]
\begin{center}
\includegraphics[scale=1]{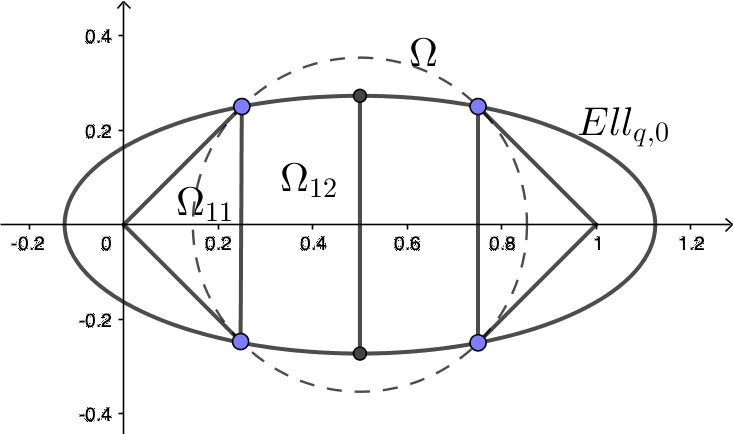}
\end{center}
\caption{The decomposition of $\Omega_q$.\label{Fig3}}
\end{figure}

Let us start with $z\in\Omega_{11}$ and hence $\Re z\in (0,1/4)$.
Then $x_1=2a-\sqrt{2(a^2-b^2)}<1/2$ and the midpoint of $[x_1,x_2]$ is $2a<1/2$, so that using \eqref{fct.h},
\begin{eqnarray*}
K_{1/2}(z)&=&\frac{1}{\pi}\int_{x_1}^{\min(x_2,1/2)}\frac{1}{\sqrt{x-x_1}}\frac{1}{\sqrt{x_2-x}}dx\\
&\ge& \frac{1}{\pi}\int_{x_1}^{2a}\frac{1}{\sqrt{x-x_1}}\frac{1}{\sqrt{x_2-x}}dx\\
&=& \frac{1}{\pi}\int_{-1}^0\frac{1}{\sqrt{1-y^2}}dy= \frac{1}{2}.
\end{eqnarray*}
Hence we have
\begin{eqnarray*}
\sum_{n=0}^{\infty}\frac{2^n}{(2n+1)!}\int_0^{1/2}
 x^{2n+1}|f^{(n)}(x)|^2dx
&=& \frac{\sqrt{2}}{2\pi}\int_0^{1/2} \int_{D_x}|f(z)|^2  \left(\frac{x^2}{{2}}-|z-x|^2\right)^{-1/2}dA(z){dx}\\
&=& \frac{\sqrt{2}}{2\pi}\int_{\Omega}|f(z)|^2\int_{0}^{1/2}\chi_{D_x}(z) \left(\frac{x^2}{{2}}-|z-x|^2\right)^{-1/2}{dx} dA(z)\\
&\ge& \frac{1}{2}\int_{\Omega_{11}}|f(z)|^2dA(z).
\end{eqnarray*}

We consider next $\Omega_{12}$. The idea is to show that when $z\in \Omega_{12}$,
the interval on which we integrate in 
\begin{eqnarray}\label{lower}
\int_{0}^{1/2}\chi_{D_x}(z) \left(\frac{x^2}{{2}}-|z-x|^2\right)^{-1/2}{dx}
\end{eqnarray}
is large enough to obtain the desired estimate.
Keeping in mind that $z\in D_x$ if and only if $x\in (x_1,x_2)$ ($x_1$ and
$x_2$ as introduced earlier), we have to ensure that $(x_1,x_2)$ meets $(0,1/2)$ on a
``sufficiently large'' interval. Note that for $\Re z\in (1/4,1/2)$, the
midpoint of $(x_1,x_2)$ is $2\Re z>1/2$, so outside the interval $(0,1/2)$. 
Still, when $z\in \Omega_1$, we have
$$
 x_1=2a-\sqrt{2(a^2-b^2)}<1/2,
$$
while $x_2=2a+\sqrt{2(a^2-b^2)}>1/2$ so that
$$
\int_{0}^{1/2}\chi_{D_x}(z) \left(\frac{x^2}{{2}}-|z-x|^2\right)^{-1/2}{dx}
=\int_{x_1}^{1/2} \left(\frac{x^2}{{2}}-|z-x|^2\right)^{-1/2}{dx}
$$
Using arguments already applied in Section 2 and the corresponding change of variables 
$$
 y(x)=2\frac{x-x_1}{x_2-x_1}-1,
$$ 
we get
\begin{align}
\int_{x_1}^{1/2}\left(\frac{x^2}{{2}}-|z-x|^2\right)^{-1/2}{dx}
&=\sqrt{2}\int_{x_1}^{1/2}\frac{1}{\sqrt{x-x_1}}\frac{1}{\sqrt{x_2-x}}dx \notag\\
&=\sqrt{2}\int_{-1}^{y(1/2)}\frac{1}{\sqrt{1-y^2}}dy. 
\label{eq-integral}
\end{align}

The proof will be complete if we can show that $y(1/2)$ is uniformly bounded from below by a 
constant strictly greater than $-1$ when $z\in \Omega_{12}$, i.e. 
$$
 y(1/2)-(-1)=y(1/2)+1\ge \delta>0, \quad z\in\Omega_{12}.
$$ 
Since for fixed $a=\Re z\in
(1/4,1/2)$, the interval $(x_1,\frac{1}{2})$ is smallest when $b=\Im z$ is largest (this is also clear from the formula defining $x_1$), it is sufficient
to do the estimate for $z=a+ib\in \partial(\Ell_{q,0}\cap\Omega_{12})$, i.e.
\begin{equation}\label{EllipseForm}
 q(a-\frac{1}{2})^2+(16-q)b^2=1,\quad 1/4<a<1/2.
\end{equation}
We have to show that 
$$
 y(1/2)+1=2\frac{1/2-x_1}{x_2-x_1}=
\frac{1/2-x_1}{\sqrt{2(a^2-b^2)}} \ge c >0.
$$
For $z\in\partial\Ell_{q,0} \cap \partial \Omega_{12}$, we parametrize $a=\frac{1}{4}+\delta$ with $0<\delta<1/4$. From \eqref{EllipseForm}, we get
\begin{equation}
\label{eq-b2}
    b^2=b(a)^2=\frac{1-q(a-1/2)^2}{16-q}.
\end{equation}
Then 
$$
\frac12-x_1=\frac{1}{2}-(2a-\sqrt{2(a^2-b^2)})= -2\delta + \sqrt{2(a^2-b^2)}=:-2\delta + t,$$
so that 
$$y(1/2)+1= -2\frac{\delta}{t} + 1. $$
With \eqref{eq-b2} in mind, setting 
$$
P(a):=t^2=2(a^2-b(a)^2)
=2\left(a^2-\frac{1-q(a-1/2)^2}{16-q}\right)
=\frac{2}{16-q}(16a^2-qa+q/4-1),
$$ 
it is easy to check that $P(1/4)=0$. 
Actually, 
$$
P(a)=\frac{2}{16-q}(a-1/4)(16a+4-q)=\frac{2}{16-q}\delta(16\delta+8-q).
$$ 
Observe that since $\delta>0$ and $q<8$, we have $(16\delta+8-q)>0$. Since $q<8$, we conclude that 
$$
y\left(\frac12\right)+1=-2\frac{\delta}{t} + 1 = -\frac{2\delta}{\sqrt{\delta} \sqrt{16\delta+8-q}}\sqrt{\frac{16-q}{2}}.
$$
It can be easily checked that $\sqrt{\frac{\delta}{16\delta+8-q}}$ is increasing in $\delta$ independently of $q\in(0,8)$, so that the smallest value of $y(1/2)+1$ is reached for $\delta=1/4$. Replacing $\delta$ by $1/4$, and observing that $(16-q)/(12-q)$ is strictly increasing in $q$, we see that
$$
y\left(\frac12\right)+1\ge 
-\sqrt{\frac{16-q}{2\times(12-q)}}+1>0,
$$
whenever $0<q<8$ (the expression is equal to 0 when $q=8$ as expected from Proposition \ref{Prop-counterexample}).
\end{proof}

{\bf Remark:} Here is another observation. It is possible to replace 
$D_x$ by $$
D_x^a=\{z\in \C:|z-x|<r_x^a:=x\frac{a}{\sqrt{1+a^2}}\}
$$
for some $a\ge 1$ (note that $D_x=D_x^1$). This implies that $\bigcup_{x>0}D_x^a$ is the sector between the lines $y=-ax$ and $y=ax$, $x>0$. Define now
$$
\Omega^a=\bigcup_{0<x\le 1/2}D_x^a\cup \bigcup_{0<x<1/2}\left(1-D_x^a\right).
$$
Figure \ref{Fig4} shows the relation between $\Omega$, $\Omega^a$ and $D$, and in particular $\Omega\subsetneq \Omega^a$ for every $a>1$.
\begin{figure}[h!]
\begin{center}
\includegraphics[scale=1.5]{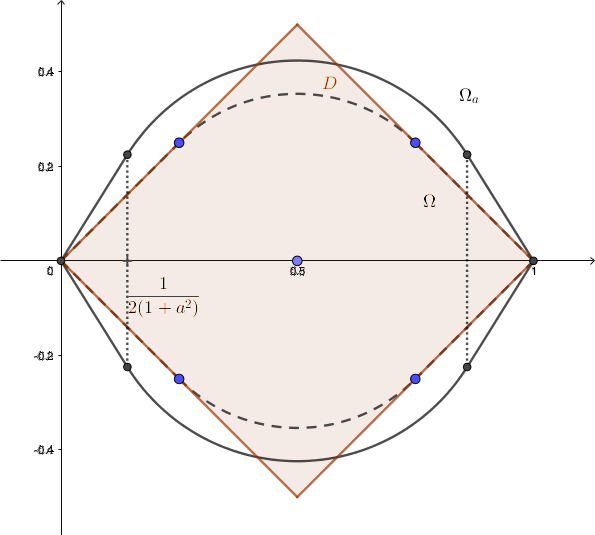}
\end{center}
\caption{The domains $\Omega$, $\Omega^a$ and $D$.\label{Fig4}}
\end{figure}
In the spirit of Proposition \ref{Prop1}, it is possible to establish estimates between a formula of type
$$\int_0^{1/2}\sum_{n=0}^{\infty}\left| \frac{(xa/\sqrt{1+a^2})^{n}f^{(n)}(x)}{n!} \right|^2
\frac{2^{2n}\,(n!)^{2}}{(2n+1)!}xdx
=\sum_{n=0}^{\infty}\frac{(2a/\sqrt{1+a^2})^{2n}}{(2n+1)!}\int_0^{1/2}
 x^{2n+1}|f^{(n)}(x)|^2dx,
$$
suitably symmetrized on $(0,1)$,
and the Bergman norm on $\Omega^a$. Here we do not want to enter into details of this discussion, but rather make the reader observe that the domain $\Omega_a$ will never cover $D$. Indeed, in order to reach for instance the upper corner $1/2+i/2$ of $D$, the disk centered at $1/2$ should have a radius $1/2$. However, $r_{1/2}^a=a/(2\sqrt{1+a^2})<1/2$ for every $a\ge 1$ so that the closure of $D_{1/2}^a$ never reaches $1/2+i/2$. As a result, even if we modify the weight of the integration in this specific way, we will not be able to recover the space $A^2(D)$ from the diagonal.

\section{Recovering the norm from two diagonals}\label{S4}

We are now in a position to prove our main result which we state here again. 

\begin{thm}
For every function $f\in A^2(D)$ we have
\begin{eqnarray}\label{MainEstim1}
 \int_D |f(z)|^2dA(z)&=& 
\sum_{n=0}^{\infty}\frac{2^n}{(2n+1)!} \int_0^{1/2}
 x^{2n+1}\left(|f^{(n)}(x)|^2+|f^{(n)}(1-x)|^2+ \right. \nonumber\\
 & &\quad\left. |f^{(n)}(1/2+i(1/2-x))|^2+|f^{(n)}(1/2+i(x-1/2))|^2
\right)dx 
\end{eqnarray}
Moreover, if $f$ is a $\mathcal{C}^\infty$-function on $(0,1)\cup (1/2-i/2,1/2+i/2)$ with $i^n \frac{\partial^n f}{\partial x^n}(1/2)= \frac{\partial^n f}{\partial y^n}(1/2)$ for all $n \in \N$,
and for which the right hand side of \eqref{MainEstim1} is finite, then $f$ extends to $A^2(D)$ and we have \eqref{MainEstim1}.
\end{thm}

\begin{proof}
We write again $z=a+ib$ and 
$$
t=\sqrt{2(a^2-b^2)},\qquad x_1=2a-t,\qquad x_2=2a+t.
$$
As previously, the proof is based on estimates of
\begin{align}
K_{1/2}(z)&=\frac{\sqrt2}{2\pi}
\int_{0}^{1/2}\chi_{D_x}(z)
\left(\frac{x^2}{2}-|z-x|^2\right)^{-1/2}\,dx \label{eq-K1}\\
&=\frac{\sqrt2}{2\pi} \int_{(x_1,\,  x_2) \cap (0, \, 1/2)}  \left(\frac{x^2}{2}-|z-x|^2\right)^{-1/2}\,dx.
\end{align}
For later purposes it is useful to introduce the following decomposition: $D=D_{1/2}\cup\bigcup_{k=1}^4\Gamma_k$ (see Figure \ref{Fig5}).
Note that $K_{1/2}(z)=1$ if $z\in\Gamma_1$ (i.e.\ $0<x_2\le 1/2$) by equation \eqref{IntChi}.
Moreover, if $z\in \Gamma_k$, $k=2,3,4$, then $z\notin \bigcup_{0<x \le 1/2} D_x$ so $\chi_{D_x}(z)=0$ for all $0<x\le 1/2$. Therefore $K_{1/2}(z)=0$ by \eqref{eq-K1}.   

\begin{figure}[h!]
\begin{center}
\includegraphics[scale=1]{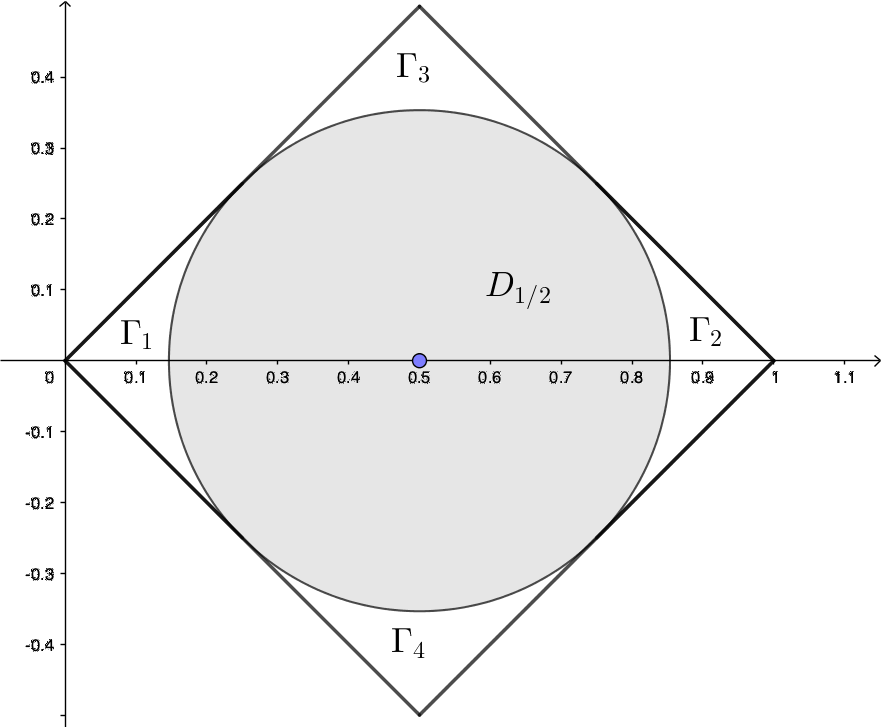}
\end{center}
\caption{The corners of $D$.\label{Fig5}}
\end{figure}

Finally, if $z \in D_{1/2}$ (i.e.\ $x_1 < 1/2 < x_2$), using again the change of variables
$y=2(x-x_1)/(x_2-x_1)-1=(x-2a)/t$, 
we get
\begin{align}
K_{1/2}(z)&=\frac{\sqrt2}{2\pi} \int_{x_1}^{1/2}  \left(\frac{x^2}{2}-|z-x|^2\right)^{-1/2}\,dx =\frac{1}{\pi}\int_{-1}^{y(1/2)} \frac{1}{\sqrt{1-y^2}} dy \notag\\
&=\frac{1}{\pi}\left(\arcsin(y(1/2))+\frac{\pi}{2}\right)=\frac{1}{\pi}\arccos(-y(1/2))\notag\\
&=\frac{1}{\pi} \arccos\left(\frac{2a-\frac{1}{2}}{t}\right) \label{eq-K}.
\end{align}

We are now ready to establish the identity \eqref{MainEstim1} for an arbitrary
$f\in\operatorname{Hol}(D)$. Set
$$
H(f)=\sum_{n\ge0}\frac{2^n}{(2n+1)!}
\int_0^{1/2}x^{2n+1}
\bigl(|f^{(n)}(x)|^2+|f^{(n)}(1-x)|^2\bigr)\,dx
$$
and define
$
w(z)=K_{1/2}(z)+K_{1/2}(1-z).
$
By equation \eqref{IdOmega} and symmetry,
$$
H(f)=\int_{\Omega} |f(z)|^2 w(z)\,dA(z)=\int_D |f(z)|^2 w(z)\,dA(z), 
$$
where we have used that $K_{1/2}$ vanishes in $D\setminus\Omega$.
Now, let
$
R(z)=\frac{1}{2}+i\left(z-\frac{1}{2}\right)
$
be the rotation by $\pi/2$ centered at $1/2$.
Since $w(1-z)=w(z)$ and $R^{-1}(z)=1-R(z)$, a change of variables gives
$$
H(f)+H(f\circ R)
=\int_D|f(z)|^2\bigl(w(z)+w(R(z))\bigr)\,dA(z).
$$
Note that $H(f)+H(f\circ R)$ is precisely the right-hand side in \eqref{MainEstim1},
as $|(f\circ R)^{(n)}|=|f^{(n)}\circ R|$.
Therefore, it suffices to prove that, for all $z \in D$,
$$
w(z)+w(R(z))=1.
$$
Suppose first that $z$ belongs to the central disk $D_{1/2}$ and denote
$$
\alpha=\alpha(z):=\sqrt{t^2 - (2a-\frac{1}{2})^2}=\sqrt{2(a^2-b^2) - (2a-\frac{1}{2})^2}.
$$
Then, using the formula 
$$\arccos(x)= \arg\left(x + i \sqrt{1-x^2)}\right), \qquad x\in (-1, \, 1),
$$
we get from \eqref{eq-K}
$$
K_{1/2}(z)=\frac{1}{\pi} \arg\left(2a-\frac{1}{2}+i \alpha \right),
$$
where the argument belongs to $(0,\pi)$ (since the imaginary part is $\alpha>0$).
Now, note that 
$$\alpha(z)=\sqrt{\frac{1}{4}-2\left(a-\frac{1}{2}\right)^2-2b^2}=\sqrt{\frac{1}{4}-2\left|z-\frac{1}{2}\right|^2}.$$
So $\alpha(z)$ only depends on the distance from $z$ to $1/2$. Therefore it is preserved by the symmetry $z \mapsto 1-z$ and the rotation $R$: $$\alpha(z)=\alpha(1-z)=\alpha(R(z)), \qquad z \in D_{1/2}.$$
This yields
$$
\begin{aligned}
\pi w(z)
=\pi \left(K_{1/2}(z) + K_{1/2}(1-z)\right)&=\arg\left(2a-\frac{1}{2}+i \alpha \right)
+\arg\left(\frac{3}{2}-2 a+i \alpha \right)\\
&=\arg\left(2\left[b^2-\left(a-\frac{1}{2}\right)^2\right]+i \alpha \right).
\end{aligned}
$$
Indeed, both factors and their product lie in the upper half-plane, so the
sum of the two arguments belongs to $(0,\pi)$ and there is no
ambiguity modulo $2\pi$.
Therefore, using that $R(z)=(1/2-b)+i(a-1/2)$, we get
$$ 
\pi w(R(z))
=\arg\left(-2\left[b^2-\left(a-\frac12\right)^2\right]+i \alpha \right).
$$ 
Since the two complex numbers in the arguments above have opposite real part and same positive imaginary part, their arguments are supplementary. Thus $ \pi w(z) + \pi w(R(z))= \pi$,
proving the required identity on $D_{1/2}$.

For $z\in D\setminus D_{1/2}$, the four points $z,1-z,R(z),1-R(z)$ belong to different $\Gamma_k$. The point $z_0$ among them which belongs to the left corner $\Gamma_1$ satisfies $K_{1/2}(z_0)=1$. The three other points 
cancel $K_{1/2}$. Thus $w(z)+w(R(z))=1$ on this region as well.
We conclude that
$$
H(f)+H(f\circ R)=\int_D|f(z)|^2\,dA(z),
$$
which proves the asserted identity, in particular for $f\in A^2(D)$.

\medskip

Finally, suppose the right-hand side of \eqref{MainEstim1} is finite. Denote by $\Omega_q'=R(\Omega_q)$ the rotated domain. Then $f$ extends analytically to $f_1$ on $\Omega_q$ and to $f_2$ on $\Omega_q'$ by Proposition \ref{Prop2}. By assumption, their Taylor coefficients coincide at $1/2$, so that $f_1=f_2$ on $\Omega_q \cap \Omega_q'$. Hence $f$ extends holomorphically to the whole domain $D=\Omega_q \cup \Omega_q'$. By the first part of the proof we conclude that this extension is in $A^2(D)$ and satisfies \eqref{MainEstim1}.
\end{proof}

\bibliographystyle{alpha}
\bibliography{biblio}

@article {ACHK,
    AUTHOR = {Aadi, D. and Cruz, C. and Hartmann, A. and Kellay, K.},
     TITLE = {Multiple sampling and interpolation in {B}ergman spaces},
   JOURNAL = {J. Funct. Anal.},
  FJOURNAL = {Journal of Functional Analysis},
    VOLUME = {284},
      YEAR = {2023},
    NUMBER = {9},
     PAGES = {Paper No. 109865, 52},
      ISSN = {0022-1236},
   MRCLASS = {30J99 (30H20 46E22 47B32)},
  MRNUMBER = {4549582},
       DOI = {10.1016/j.jfa.2023.109865},
       URL = {https://doi.org/10.1016/j.jfa.2023.109865},
}

@article{AHS,
	Author = {Aikawa, H. AND Hayashi, N. AND Saitoh, S.},
	Journal = {Complex Variables Theory Appl.},
	Volume = {15},
	Pages = {pp. 27-36},
	Title = {{T}he {B}ergman space on a sector and the heat equation},
	Year = {1990}}

@article {AHSb,
    AUTHOR = {Aikawa, H. and Hayashi, N. and Saitoh, S.},
     TITLE = {Isometrical identities for the {B}ergman and the {S}zeg\"o{} spaces on a sector},
   JOURNAL = {J. Math. Soc. Japan},
  FJOURNAL = {Journal of the Mathematical Society of Japan},
    VOLUME = {43},
      YEAR = {1991},
    NUMBER = {1},
     PAGES = {195--201},
      ISSN = {0025-5645,1881-1167},
   MRCLASS = {30H05 (46E15)},
  MRNUMBER = {1082430},
MRREVIEWER = {Manfred\ W.\ Kracht},
       DOI = {10.2969/jmsj/04310195},
       URL = {https://doi-org.acces.bibl.ulaval.ca/10.2969/jmsj/04310195},
}

@article {BHKM,
    AUTHOR = {Borichev, A. and Hartmann, A. and Kellay, K. and Massaneda, X.},
     TITLE = {Geometric conditions for multiple sampling and interpolation in the {F}ock space},
   JOURNAL = {Adv. Math.},
  FJOURNAL = {Advances in Mathematics},
    VOLUME = {304},
      YEAR = {2017},
     PAGES = {1262--1295},
      ISSN = {0001-8708},
   MRCLASS = {30H20 (30E05 46C07 46E22 47B32 47B35)},
  MRNUMBER = {3558232},
MRREVIEWER = {Sei-ichiro Ueki},
       DOI = {10.1016/j.aim.2016.09.019},
       URL = {https://doi.org/10.1016/j.aim.2016.09.019},
}

@article {DE,
    AUTHOR = {Dard\'e, J. and Ervedoza, S.},
     TITLE = {On the reachable set for the one-dimensional heat equation},
   JOURNAL = {SIAM J. Control Optim.},
  FJOURNAL = {SIAM Journal on Control and Optimization},
    VOLUME = {56},
      YEAR = {2018},
    NUMBER = {3},
     PAGES = {1692--1715},
      ISSN = {0363-0129,1095-7138},
   MRCLASS = {93B03 (35K05 35K20 93C20)},
  MRNUMBER = {3802267},
MRREVIEWER = {Larissa\ V.\ Fardigola},
       DOI = {10.1137/16M1093215},
       URL = {https://doi-org.acces.bibl.ulaval.ca/10.1137/16M1093215},
}

@article {FR,
    AUTHOR = {Fattorini, H. O. and Russell, D. L.},
     TITLE = {Exact controllability theorems for linear parabolic equations in one space dimension},
   JOURNAL = {Arch. Rational Mech. Anal.},
  FJOURNAL = {Archive for Rational Mechanics and Analysis},
    VOLUME = {43},
      YEAR = {1971},
     PAGES = {272--292},
      ISSN = {0003-9527},
   MRCLASS = {93B05},
  MRNUMBER = {335014},
MRREVIEWER = {F.\ M.\ Kirillova},
       DOI = {10.1007/BF00250466},
       URL = {https://doi-org.acces.bibl.ulaval.ca/10.1007/BF00250466},
}

@incollection {FHR,
    AUTHOR = {Fricain, E. and Hartmann, A. and Ross, W. T.},
     TITLE = {A survey on reverse {C}arleson measures},
 BOOKTITLE = {Harmonic analysis, function theory, operator theory, and their
              applications},
    SERIES = {Theta Ser. Adv. Math.},
    VOLUME = {19},
     PAGES = {91--123},
 PUBLISHER = {Theta, Bucharest},
      YEAR = {2017},
   MRCLASS = {30J05 (30H10 30H20 46E22)},
  MRNUMBER = {3753895},
MRREVIEWER = {E. S. Dubtsov},
}

@book {HKZ,
    AUTHOR = {Hedenmalm, H. and Korenblum, B. and Zhu, K.},
     TITLE = {Theory of {B}ergman spaces},
    SERIES = {Graduate Texts in Mathematics},
    VOLUME = {199},
 PUBLISHER = {Springer-Verlag, New York},
      YEAR = {2000},
     PAGES = {x+286},
      ISBN = {0-387-98791-6},
   MRCLASS = {46E15 (30-02 30H05 46J15 47B38)},
  MRNUMBER = {1758653},
MRREVIEWER = {D. Sarason},
       DOI = {10.1007/978-1-4612-0497-8},
       URL = {https://doi.org/10.1007/978-1-4612-0497-8},
}

@article {HKT,
    AUTHOR = {Hartmann, A. and Kellay, K. and Tucsnak, M.},
     TITLE = {From the reachable space of the heat equation to {H}ilbert
              spaces of holomorphic functions},
   JOURNAL = {J. Eur. Math. Soc. (JEMS)},
  FJOURNAL = {Journal of the European Mathematical Society (JEMS)},
    VOLUME = {22},
      YEAR = {2020},
    NUMBER = {10},
     PAGES = {3417--3440},
      ISSN = {1435-9855},
   MRCLASS = {93B05 (30H10 30H15 30H20 35K05)},
  MRNUMBER = {4153111},
MRREVIEWER = {Mohammed El A\"{\i}di, Universidad Nacional de Colombia},
       DOI = {10.4171/jems/989},
       URL = {https://doi.org/10.4171/jems/989},
}

@article {HO,
    AUTHOR = {Hartmann, A. and Orsoni, M.-A.},
     TITLE = {Separation of singularities for the {B}ergman space and
              application to control theory},
   JOURNAL = {J. Math. Pures Appl. (9)},
  FJOURNAL = {Journal de Math\'{e}matiques Pures et Appliqu\'{e}es. Neuvi\`eme S\'{e}rie},
    VOLUME = {150},
      YEAR = {2021},
     PAGES = {181--201},
      ISSN = {0021-7824},
   MRCLASS = {30H20 (35K05 46E15 93B03)},
  MRNUMBER = {4248466},
MRREVIEWER = {Serhii V. Gryshchuk},
       DOI = {10.1016/j.matpur.2021.04.009},
       URL = {https://doi.org/10.1016/j.matpur.2021.04.009},
}

@article {HS,
    AUTHOR = {Hayashi, N. and Saitoh, S.},
     TITLE = {Analyticity and smoothing effect for the {S}chr\"odinger
              equation},
   JOURNAL = {Ann. Inst. H. Poincar\'e{} Phys. Th\'eor.},
  FJOURNAL = {Annales de l'Institut Henri Poincar\'e. Physique Th\'eorique},
    VOLUME = {52},
      YEAR = {1990},
    NUMBER = {2},
     PAGES = {163--173},
      ISSN = {0246-0211},
   MRCLASS = {35J10 (35B65 81Q05)},
  MRNUMBER = {1051235},
MRREVIEWER = {Volker\ Vogelsang},
       URL = {http://www.numdam.org/item?id=AIHPA_1990__52_2_163_0},
}

@article {KNT,
    AUTHOR = {Kellay, K. and Normand, T. and Tucsnak, M.},
     TITLE = {Sharp reachability results for the heat equation in one space dimension},
   JOURNAL = {Anal. PDE},
  FJOURNAL = {Analysis \& PDE},
    VOLUME = {15},
      YEAR = {2022},
    NUMBER = {4},
     PAGES = {891--920},
      ISSN = {2157-5045},
   MRCLASS = {93B03 (30H20 35K08 93B05 93C20)},
  MRNUMBER = {4478293},
MRREVIEWER = {Michela Egidi},
       DOI = {10.2140/apde.2022.15.891},
       URL = {https://doi.org/10.2140/apde.2022.15.891},
}

@article {Lu,
    AUTHOR = {Luecking, D. H.},
     TITLE = {Forward and reverse {C}arleson inequalities for functions in
              {B}ergman spaces and their derivatives},
   JOURNAL = {Amer. J. Math.},
  FJOURNAL = {American Journal of Mathematics},
    VOLUME = {107},
      YEAR = {1985},
    NUMBER = {1},
     PAGES = {85--111},
      ISSN = {0002-9327},
   MRCLASS = {30A10 (30E05 46E15)},
  MRNUMBER = {778090},
MRREVIEWER = {Charles Horowitz},
       DOI = {10.2307/2374458},
       URL = {https://doi.org/10.2307/2374458},
}

@article {MRR,
    AUTHOR = {Martin, P. and Rosier, L. and Rouchon, P.},
     TITLE = {On the reachable states for the boundary control of the heat
              equation},
   JOURNAL = {Appl. Math. Res. Express. AMRX},
  FJOURNAL = {Applied Mathematics Research Express. AMRX},
      YEAR = {2016},
    NUMBER = {2},
     PAGES = {181--216},
      ISSN = {1687-1200,1687-1197},
   MRCLASS = {93B03 (35K05 35K20 93C20)},
  MRNUMBER = {3551775},
MRREVIEWER = {Larissa\ V.\ Fardigola},
       DOI = {10.1093/amrx/abv013},
       URL = {https://doi.org/10.1093/amrx/abv013},
}

@article {O,
    AUTHOR = {Orsoni, M.-A.},
     TITLE = {Reachable states and holomorphic function spaces for the 1-{D}
              heat equation},
   JOURNAL = {J. Funct. Anal.},
  FJOURNAL = {Journal of Functional Analysis},
    VOLUME = {280},
      YEAR = {2021},
    NUMBER = {7},
     PAGES = {Paper No. 108852, 17},
      ISSN = {0022-1236},
   MRCLASS = {93B03 (30H20 35K05 44A10)},
  MRNUMBER = {4211029},
MRREVIEWER = {Nino Manjavidze},
       DOI = {10.1016/j.jfa.2020.108852},
       URL = {https://doi.org/10.1016/j.jfa.2020.108852},
}

@book{QQ,
	Alteditor = {*},
	Author = {Queff{\'e}lec, H. AND Queff{\'e}lec, M.},
	Optaddress = {*},
	Optannote = {*},
	Optedition = {*},
	Optkey = {*},
	Optmonth = {*},
	Optnote = {*},
	Optnumber = {*},
	Optseries = {*},
	Optvolume = {*},
	Publisher = {Calvage {\&} {M}ounet},
	Title = {Analyse complexe et applications},
	Year = {2017}}

@book {Seip,
    AUTHOR = {Seip, K.},
     TITLE = {Interpolation and sampling in spaces of analytic functions},
    SERIES = {University Lecture Series},
    VOLUME = {33},
 PUBLISHER = {American Mathematical Society, Providence, RI},
      YEAR = {2004},
     PAGES = {xii+139},
      ISBN = {0-8218-3554-8},
   MRCLASS = {30E05 (30D45 30D55 46E15 46E20 47B35 94A20)},
  MRNUMBER = {2040080},
MRREVIEWER = {Richard Rochberg},
       DOI = {10.1090/ulect/033},
       URL = {https://doi.org/10.1090/ulect/033},
}

@article {SW,
    AUTHOR = {Seip, K. and Wallst\'{e}n, R.},
     TITLE = {Density theorems for sampling and interpolation in the
              {B}argmann-{F}ock space. {II}},
   JOURNAL = {J. Reine Angew. Math.},
  FJOURNAL = {Journal f\"{u}r die Reine und Angewandte Mathematik. [Crelle's
              Journal]},
    VOLUME = {429},
      YEAR = {1992},
     PAGES = {107--113},
      ISSN = {0075-4102},
   MRCLASS = {46E20 (30E05 30H05 46E22)},
  MRNUMBER = {1173118},
MRREVIEWER = {Richard Rochberg},
       DOI = {10.1515/crll.1992.429.107},
       URL = {https://doi.org/10.1515/crll.1992.429.107},
}

@article {SeI,
    AUTHOR = {Seip, K.},
     TITLE = {Density theorems for sampling and interpolation in the
              {B}argmann-{F}ock space. {I}},
   JOURNAL = {J. Reine Angew. Math.},
  FJOURNAL = {Journal f\"{u}r die Reine und Angewandte Mathematik. [Crelle's
              Journal]},
    VOLUME = {429},
      YEAR = {1992},
     PAGES = {91--106},
      ISSN = {0075-4102},
   MRCLASS = {46E20 (30E05 30H05 46E22)},
  MRNUMBER = {1173117},
MRREVIEWER = {Richard Rochberg},
       DOI = {10.1515/crll.1992.429.91},
       URL = {https://doi.org/10.1515/crll.1992.429.91},
}

\end{document}